\documentclass[11pt]{amsart}
\usepackage{amsmath,amssymb,latexsym,esint,cite,mathrsfs}
\usepackage{verbatim,wasysym}
\usepackage[left=2.65cm,right=2.65cm,top=3cm,bottom=3cm]{geometry}

\usepackage{microtype}
\usepackage{color,enumitem,graphicx}
\usepackage[colorlinks=true,urlcolor=blue, citecolor=red,linkcolor=blue,
linktocpage,pdfpagelabels, bookmarksnumbered,bookmarksopen]{hyperref}
\usepackage[hyperpageref]{backref}
\usepackage[english]{babel}

\usepackage{appendix}

\def\R{{\mathbb {R}}}
\def\N{{\mathbb {N}}}

\def\ve{\varepsilon}
\def\cd{\rightharpoonup}

\def\supp{\operatorname {\text{supp}}}

\def\dist{\operatorname {\text{dist}}}

\newtheorem{teo}{Theorem}[section]
\newtheorem{lema}[teo]{Lemma}
\newtheorem{prop}[teo]{Proposition}
\newtheorem{corol}[teo]{Corollary}

\theoremstyle{remark}

\theoremstyle{definition}

\numberwithin{equation}{section}

\begin{document}

\title[On the solvability of a minimization problem]{A minimization problem involving a  fractional \\ Hardy-Sobolev type inequality}

\author{Antonella Ritorto}

\address{Mathematical Institute, Utrecht University \newline
	Hans Freudenthalgebouw \newline
	Budapestlaan 6, 3584 CD Utrecht, Netherlands}
\email[A. Ritorto]{a.ritorto@uu.nl}
	
\subjclass[2010]{35R11, 35R45}
	
\keywords{Fractional Hardy-Sobolev type inequality, Minimization problem}

\begin{abstract} In this work, we obtain an existence of nontrivial solutions to a minimization problem involving a fractional Hardy-Sobolev type inequality in the case of inner singularity. Precisely, for $\lambda>0$ we analyze the attainability of the optimal constant
$$
\mu_{\alpha, \lambda}(\Omega):=\inf\left\{ [u]^2_{s,\Omega}+\lambda\int_{\Omega}|u|^2 \, dx \colon u\in H^s(\Omega), \,  \int_{\Omega} \frac{|u(x)|^{2_{s,\alpha}}}{|x|^{\alpha}} \, dx=1  \right\},
$$	
where $0<s<1, n>4s, 0<\alpha<2s$, $2_{s,\alpha}=\frac{2(n-\alpha)}{n-2s}$, and $\Omega \subset \R^n$ be a bounded domain such that $0\in \Omega$. 
\end{abstract}
	
\maketitle

\section{Introduction}


Let $0<s<1$, $n>4s$, $0<\alpha<2s$,  and $\Omega \subset \R^n$ be a bounded domain such that $0\in \Omega$. We introduce the fractional Sobolev space, see for instance \cite{DiNezza-Palatucci-Valdinoci}, 
\begin{equation}\label{Hs}
H^s(\Omega):=\left\{ u\in L^2(\Omega) \colon \frac{|u(x)-u(y)|}{|x-y|^{\frac{n}{2}+s}} \in L^2(\Omega\times \Omega) \right\},
\end{equation}
endowed with the norm
\begin{equation}\label{normaHs}
\|u\|_{s,\Omega}:=\left( \int_{\Omega}|u|^2\, dx + \int_{\Omega\times \Omega} \frac{|u(x)-u(y)|^2}{|x-y|^{n+2s}}\, dxdy\right)^{\frac{1}{2}}.
\end{equation}
Let $\lambda>0$ and $2_{s,\alpha}= \frac{2(n-\alpha)}{n-2s}$. This paper concerns in analyzing the attainability of the optimal constant $C>0$ for the following fractional Hardy-Sobolev inequality
$$
C\left(\int_{\Omega} \frac{|u(x)|^{2_{s,\alpha}}}{|x|^{\alpha}} \, dx \right)^{\frac{2}{2_{s,\alpha}}} \le \int_{\Omega \times \Omega} \frac{|u(x)-u(y)|^2}{|x-y|^{n+2s}} \, dxdy + \lambda \int_{\Omega} |u(x)|^2 \, dx
$$  
for every $u \in H^s(\Omega)$. For the related Dirichlet problem see the recent work \cite{Ghoussoub-Robert-Shakerian-Zhao}. 


In \cite{Marano-Mosconi}, S. A. Marano and S. Mosconi prove the existence of an extremal function $u_0$, solution to 
\begin{equation}\label{global-mu}
\mu_{\alpha}:= \inf\left\{ [u]_s^2 \colon u \text{ measurable, vanishing at infinity}, \quad  \int_{\R^N}\frac{|u(x)|^{2_{s,\alpha}}}{|x|^{\alpha}} \, dx=1 \right\}.
\end{equation}
where $2_{s,\alpha}=\frac{2(n-\alpha)}{n-2s}$ and 
$$
[u]_s^2 = \int_{\R^n \times \R^n} \frac{|u(x)-u(y)|^2}{|x-y|^{n+2s}} \, dxdy.
$$
See also \cite{Yang}. Here, $u$ vanishes at infinity means $\left| \{ |u|> a \} \right| < \infty \text{ for every } a\in \R$. Observe that  $2_{s,2s}=2$ and $2_{s,0}=2_s^*=\frac{2n}{n-2s}$, the latter is related to the non compact but continuous embedding $H^s(\R^n)\hookrightarrow L^{2_s^*}(\R^n)$. The constant $\mu_{2s}$  was calculated by I. Herbst \cite{Herbst1977}. 
In \cite{Marano-Mosconi}, for $p>1$, the existence of extremal functions $u \in W^{s,p}(\R^n)$ for the Hardy-Sobolev inequality is established through concentration-compactness. The authors also show the asymptotic behavior of extremal functions: $u(x)\sim |x|^{-\frac{n-ps}{p-1}},$ as $|x|\to \infty$, and the summability information $u\in W^{s,\gamma}(\R^n)$, for every $\frac{n(p-1)}{n-s}<\gamma<p$. Such properties turn out to be optimal when $s\to1^-$, in which case optimizers are explicitly known. See for instance \cite{DiNezza-Palatucci-Valdinoci} for the definitions of $W^{s,p}(\R^n)$ and $W^{s,\gamma}(\R^n)$. 

In \cite{Frank-Seiringer}, the sharp constant in the Hardy inequality for fractional Sobolev spaces is calculated by using a non-linear and non-local version of the ground state representation. 

For unbounded domains, different from $\R^n$, in \cite{Dyda-Lehrback-Vahakangas}, it was proved a variant of the fractional Hardy-Sobolev-Maz'ya inequality for half spaces, applying a new version of the fractional Hardy-Sobolev inequality general unbounded John domains. R. Frank and R. Seiringer give an expression for the best constant 
in the half space \cite{Frank-Seiringer-2}. See also  \cite{Bogdan-Dyda}.
Concerning bounded domains, see \cite{Dyda, Loss-Sloane}. In \cite{Edmunds-HurriSyrjanen-Vahakangas}, the authors 
consider domains with uniformly fat complement. 

In the local setting, in \cite{Ghoussoub-Kang}, the authors show that the value and the attainability of the best Hardy-Sobolev constant on a smooth domain $\Omega \subset \R^n$
$$
\nu_{\alpha}(\Omega):= \left\{ \int_{\Omega}|\nabla u|^2\, dx \colon u \in H_0^1(\Omega), \, \int_{\Omega}\frac{|u(x)|^{2_{\alpha}}}{|x|^{\alpha}}\, dx =1  \right\}
$$
are closely related to the properties of the curvature of $\partial \Omega$ at $0$, where $2_{\alpha}=\frac{2(n-\alpha)}{n-2}$, $n\ge 3, 0<\alpha<2$, when $0\in \partial \Omega$. For the non-singular context either $\alpha=0$ or $0$ belonging in the interior of the domain $\Omega$, it is well-known that $\nu_{\alpha}(\Omega)=\nu_{0}(\R^n)$ for any domain $\Omega$. 

In \cite{Hashizume}, a minimization problem involving a Hardy-Sobolev type inequality is solved, where the author analyzes both inner and boundary singularity, that is, zero belongs in the interior of the bounded domain, or  zero belongs to its boundary. For further references in the local setting, see \cite{Caffarelli-Kohn-Nirenberg, Catrina-Wang} and the expository paper \cite{Ghoussoub-Robert}.

\medskip

Our goal is analyzing the existence of solution to a minimization problem involving a fractional Hardy-Sobolev type inequality, and a positive parameter $\lambda>0$, with the inner singularity. To be precise, we first set the notation. 

From now on, we fix $0<s<1$, $n>4s$, $0<\alpha<2s$,  and $\Omega \subset \R^n$ be a bounded domain such that $0\in \Omega$. We consider the fractional Sobolev space $H^s(\Omega)$ as in \eqref{Hs}, endowed with the norm $\|u\|_{s,\Omega}$ \eqref{normaHs}, see for instance \cite{DiNezza-Palatucci-Valdinoci} for general properties. Denote 
$$
[u]_{s,\Omega}:= \left(\int_{\Omega\times \Omega} \frac{|u(x)-u(y)|^2}{|x-y|^{n+2s}}\, dxdy\right)^{\frac{1}{2}}, \quad \text{ and } \quad  \|u\|_{s,\alpha,\Omega}:= \left( \int_{\Omega} \frac{|u(x)|^{2_{s,\alpha}}}{|x|^{\alpha}} \, dx\right)^{\frac{1}{2_{s,\alpha}}}.
$$
When $\Omega=\R^n$, the notation becomes $[u]_s,\|u\|_{s,\alpha}$ respectively. We denote $\dot{H}^s(\R^n)$ the space of measurable functions $u\colon \R^n \to \R$ such that $[u]_s$ is finite.  Let $\lambda>0$ and $2_{s,\alpha}= \frac{2(n-\alpha)}{n-2s}$. Consider the following problem 
\begin{equation}\label{defi-mu-alpha-lambda}
\mu_{\alpha, \lambda}(\Omega):=\inf\left\{ [u]_{s,\Omega}^2+\lambda\int_{\Omega}|u|^2 \, dx \colon u\in H^s(\Omega), \,  \int_{\Omega} \frac{|u(x)|^{2_{s,\alpha}}}{|x|^{\alpha}} \, dx=1  \right\}
\end{equation}
We obtain the following existence results for minimizers of  \eqref{defi-mu-alpha-lambda}. 
\begin{teo}
	\label{main-theo}
Let $\lambda>0$, $0<s<1$, $n>4s$, $0<\alpha<2s$, $2_{s,\alpha}= \frac{2(n-\alpha)}{n-2s}$, and $\Omega \subset \R^n$ be a bounded domain with $0\in \Omega$. Then, there exists $\lambda_* \in (0, \infty]$ such that the constant  $\mu_{\alpha,\lambda}(\Omega)$ is attained for every $0<\lambda <\lambda_*$. Moreover, if $\lambda_*<\infty$, $\mu_{\alpha,\lambda}(\Omega)$ is not attained for every $\lambda>\lambda_*$.
\end{teo}

The rest of the paper is organize as follows. In Section \ref{preliminaries}, we gather some preliminaries and features of the constant $\mu_{\alpha,\lambda}(\Omega)$.  Section \ref{zeroinside} is dedicated to the proof of Theorem \ref{main-theo}. The crucial ingredients are the properties of $\mu_{\alpha,\lambda}(\Omega)$ seen as a function in $\lambda$ and a fractional Hardy-Sobolev type inequality. 

\section{Preliminaries}\label{preliminaries}

The relation between the global constant $\mu_{\alpha}$ and $\mu_{\alpha,\lambda}(\Omega)$, defined in \eqref{global-mu} and \eqref{defi-mu-alpha-lambda} respectively, will be a key element for the non-existence result in Theorem \ref{main-theo}.  As mentioned, some features of $\mu_{\alpha,\lambda}(\Omega)$ seen as a function in the parameter $\lambda$ play an important role as well. We start with the following basic lemma. 

\begin{lema}\label{cut-off-lemma} 
 Let  $\phi \in C_c^{\infty}(\Omega)$ and $u\in \dot{H}^s(\R^n)$ be such that $\|u\|_{s,\alpha}<\infty$, and $|u(x)|\le \frac{C}{|x|^{n-2s}}$ if $|x|\ge 1$. Then, $\phi u \in H^s(\Omega)$.
\end{lema}
\begin{proof} 
It is clear that $\phi u \in L^2(\Omega)$. Indeed, notice that $\phi u=0$ in $\R^n \setminus \Omega$, since $\supp \phi \subset \Omega$. It is clear that $\phi u \in L^2(\Omega)$, since the embedding $L^{2_{s,\alpha}}(\Omega, |x|^{-\alpha}dx)\hookrightarrow L^2(\Omega)$ is continuous, as a consequence of H\"older's inequality with $p=\frac{2_{s,\alpha}}{2}, p'=\frac{n-\alpha}{2s-\alpha}$ and the boundedness of  $\Omega$. 
	
To see $[\phi u]_{s,\Omega}<\infty$, observe that 
	\begin{equation}\label{add-subtract}
	|\phi(x)u(x)-\phi(y)u(y)| \le |u(x)||\phi(x)-\phi(y)|+|\phi(y)||u(x)-u(y)|.
	\end{equation}
	Therefore, by Minkowski's inequality,  we get
	\begin{align*}
	[\phi u]_{s,\Omega} &\le \left( \int_{\Omega} |u(x)|^2\int_{\Omega}\frac{|\phi(x)-\phi(y)|^2}{|x-y|^{n+2s}}\, dydx\right)^{\frac{1}{2}}+\left( \int_{\Omega} |\phi(x)|^2 \int_{\Omega} \frac{|u(x)-u(y)|^2}{|x-y|^{n+2s}}\, dydx\right)^{\frac{1}{2}}\\
	&=: I + C(\phi)[u]_s.
	\end{align*}
	where we have used $|\phi(x)|^2\le \|\phi\|_{\infty}^2$ in the second term. For $I$, notice that for $x\in \R^n$,
	\begin{align*}
	&\int_{\Omega}\frac{|\phi(x)-\phi(y)|^2}{|x-y|^{n+2s}}\,dy \le \int_{\R^n}\frac{|\phi(x)-\phi(y)|^2}{|x-y|^{n+2s}}\,dy  \\
	&\le \int_{|x-y|<1}\frac{|\phi(x)-\phi(y)|^2}{|x-y|^{n+2s}}\, dy+\int_{|x-y|\ge 1}\frac{|\phi(x)-\phi(y)|^2}{|x-y|^{n+2s}}\, dy	\\
	&\le \|\nabla \phi\|_{\infty}^2\int_{|x-y|<1}\frac{1}{|x-y|^{n+2s-2}}\, dy+2\|\phi\|_{\infty}^2 \int_{|x-y|\ge 1}\frac{1}{|x-y|^{n+2s}}\, dy\\
	&\le |B_1(0)|\|\nabla \phi\|_{\infty}^2\int_{0}^1\frac{r^{n-1}}{r^{n+2s-2}}\, dr+2|B_1(0)|\|\phi\|_{\infty}^2 \int_1^{\infty}\frac{r^{n-1}}{r^{n+2s}}\, dr\\
	&\le \frac{1}{2(1-s)}|B_1(0)|\|\nabla \phi\|_{\infty}^2+\frac{1}{s}|B_1(0)|\|\phi\|_{\infty}^2 =:C(\phi, n, s).
	\end{align*}
	Finally, uniformly in $x\in\R^n$, 
	\begin{equation}\label{cut-off-bound}
	\int_{\R^n}\frac{|\phi(x)-\phi(y)|^2}{|x-y|^{n+2s}}\, dy\le C(\phi, n, s).
	\end{equation}
 We split the integral and apply H\"older inequality in $|x|<1$ and the behavior of $u$ for $|x|\ge 1$, to obtain
	\begin{align*}
	\int_{\Omega} |u|^2 \, dx &\le  \int_{\Omega \cap\{ |x|<1 \}} \frac{|u|^2}{|x|^{\frac{2\alpha}{2_{s,\alpha}}}}\, dx + \int_{|x|\ge 1} \frac{1}{|x|^{2(n-2s)}}\, dx \le  C \left(\int_{\R^n} \frac{|u|^{2_{s,\alpha}}}{|x|^{\alpha}}\, dx\right)^{\frac{2}{2_{s,\alpha}}} + C <\infty,
	\end{align*}
	where we have used $n>4s$ in the second term. Hence, $[\phi u]_{s,\Omega} <\infty$, which finishes the proof of $\phi u \in H^s(\Omega)$. 
\end{proof}

Now, we are able to establish the main result of this section, which gives useful properties of $\mu_{\alpha,\lambda}(\Omega)$ seen us a function in the parameter $\lambda>0$. Part of the next Lemma relies on the existence of an extremal function for the global constant $\mu_{\alpha}$, and its behavior for $|x|\ge 1$, given in \cite{Marano-Mosconi}.

\begin{lema}\label{basic-lemma} Let $\lambda>0$ and $\Omega \subset \R^n$ be an open bounded domain such that $0\in \Omega$. 
	\begin{itemize}
		\item[(1)]  $\mu_{\alpha,\lambda}(\Omega) \le \mu_{\alpha}$, for every $\lambda>0$.
		\item[(2)] $\mu_{\alpha,\lambda}(\Omega)$ is continuous and nondecreasing with respect to $\lambda$.
		\item[(3)] $\lim_{\lambda \to 0} \mu_{\alpha,\lambda}(\Omega)=0$,
	\end{itemize}
	where $\mu_{\alpha,\lambda}(\Omega)$, and $\mu_{\alpha}$ are defined in \eqref{defi-mu-alpha-lambda}, and \eqref{global-mu} respectively.
\end{lema}
\begin{proof}
	(1) Let $\ve>0, R>0$ and $\phi \in C_c^{\infty}(\Omega)$ be such that $0\le \phi\le1$, $\phi=1$ in $B_R(0) \subset \Omega$, $\phi=0$ in $\Omega \setminus B_{2R}(0)$.  
	
	Let $u_0 $ be a positive minimizer of $\mu_{\alpha}$, see \cite{Marano-Mosconi} for the existence of $u_0$. Consider
	$$
	u_{\ve}(x):= \ve^{-\frac{n-2s}{2}} u_0\left( \frac{x}{\ve}\right) \phi(x), \qquad v_{\ve}(x):= \frac{1}{\|u_{\ve}\|_{s,\alpha,\Omega}}u_{\ve}(x).
	$$ 
	Then, $v_{\ve} \in H^s(\Omega)$, by Lemma \ref{cut-off-lemma}, since $u_0$ verifies the growth condition $|u_0(x)|\le \frac{C}{|x|^{n-2s}}$ if $|x|\ge 1$, given in \cite[Theorem 1.1]{Marano-Mosconi}. Moreover, $\|v_{\ve}\|_{s,\alpha,\Omega}=1$. Thus, 
	\begin{equation}\label{v_epsilon}
	\mu_{\alpha,\lambda}(\Omega)\le [v_{\ve}]_{s,\Omega}^2+\lambda \int_{\Omega} v_{\ve}^2(x)\, dx.
	\end{equation}
	Observe that, after a change of variables, 
	\begin{align*}
	\int_{\Omega} \frac{u_{\ve}^{2_{s,\alpha}}(x)}{|x|^{\alpha}}\, dx&=\int_{\ve^{-1}\Omega} \phi^{2_{s,\alpha}}(\ve y) \frac{u_0^{2_{s,\alpha}}(y)}{|y|^{ \alpha}} \, dy.
	\end{align*}
	Since $\phi=1$ in $B_R(0)\subset \Omega$, $0\le \phi \le 1$ and $\supp \phi \subset B_{2R}(0)$, we get 
	$$
	\int_{B_{\frac{R}{\ve}}(0)} \frac{u_0^{2_{s,\alpha}}(y)}{|y|^{ \alpha}} \, dy  \le \int_{\Omega} \frac{u_{\ve}^{2_{s,\alpha}}(x)}{|x|^{\alpha}}\, dx \le \int_{B_{\frac{2R}{\ve}}(0)}  \frac{u_0^{2_{s,\alpha}}(y)}{|y|^{ \alpha}} \, dy,
	$$
	from where we deduced
	\begin{equation}\label{a-epsilon}
	\lim_{\ve \to 0}	\int_{\Omega}\frac{u_{\ve}^{2_{s,\alpha}}(x)}{|x|^{\alpha}}\, dx=\int_{\R^n}  \frac{u_0^{2_{s,\alpha}}(y)}{|y|^{ \alpha}} \, dy=1.
	\end{equation}
	Moreover, 
	\begin{align*}
	\int_{\Omega} v_{\ve}^2(x)\, dx&= \frac{\ve^{2s-n}}{\|u_{\ve}\|_{s,\alpha,\Omega}^{2}}\,   \int_{\Omega} \phi^2(x) u_{0}\left( \frac{x}{\ve}\right)^2 \, dx =\frac{\ve^{2s}}{\|u_{\ve}\|_{s,\alpha,\Omega}^{2}}\,   \int_{B_{\frac{2R}{\ve}(0)}} \phi^2(\ve y) u_{0}\left(y\right)^2 \, dy =O(\ve^{2s}).
	\end{align*}
	The last identity is due to \eqref{a-epsilon}, and the fact that 
	\begin{equation}\label{norma-2}
	\int_{B_{\frac{2R}{\ve}(0)}} \phi^2(\ve y) u_{0}\left(y\right)^2 \, dy\le C.
	\end{equation} 
	Indeed, by \cite[Theorem 1.1]{Marano-Mosconi}, we know that for 
	\begin{equation}\label{asymp}
	|u_0(y)|\le \frac{C}{|y|^{n-2s}}, \quad \text{ for every } |y|\ge 1.
	\end{equation}
	Then, there exist $\ve_0>0$ such that for every $0<\ve<\ve_0$ we have $\frac{2R}{\ve}>1$. Therefore,  for every $0<\ve<\ve_0$,
	\begin{align*}
	\int_{B_{\frac{2R}{\ve}(0)}} \phi^2(\ve y) u_{0}\left(y\right)^2 \, dy&= \left( \int_{ \{|y|<1\}}+ \int_{ \{1\le|y|\le \frac{2R}{\ve}\}} \right) \phi^2(\ve y) u_{0}\left(y\right)^2 \, dy\\
	&=: I+II.
	\end{align*}
	To manage $I$, recall $0\le \phi\le 1$, and apply H\"older's inequality with $p=\frac{2_{s,\alpha}}{2}, p'=\frac{n-\alpha}{2s-\alpha}$, to obtain
	\begin{align*}
	I&\le C\left(\int_{ \{|y|<1\}} \frac{u_0^{2_{s,\alpha}}(y)}{|y|^{\alpha}}\, dy\right)^{\frac{2}{2_{s,\alpha}}} \le C\left(\int_{\R^n} \frac{u_0^{2_{s,\alpha}}(y)}{|y|^{\alpha}}\, dy\right)^{\frac{2}{2_{s,\alpha}}}= C.
	\end{align*}
	To control $II$, we use $0\le \phi\le 1$, \eqref{asymp} and the fact that $n>4s$, to find 
	\begin{align*}
	II&\le C\int_{|y|\ge 1}\frac{1}{|y|^{2(n-2s)}}\, dy= C\int_1^{\infty}r^{-n-1+4s}dr=C.
	\end{align*}

	Now, we have to estimate $[v_{\ve}]_{s,\Omega}^2= \|u_{\ve}\|_{s,\alpha,\Omega}^{-2}[u_{\ve}]_{s,\Omega}^2$. Thanks to \eqref{a-epsilon}, it will be enough to analyze  $[u_{\ve}]^2_{s,\Omega}$. Similar to what we have done in Lemma \ref{cut-off-lemma} (\eqref{add-subtract}, Minkowski's inequality), but changing variables and recalling  $0\le \phi\le1$, we get 
	\begin{align*}
	[u_{\ve}]_{s,\Omega}&\le [u_0]_s+\left(\int_{\ve^{-1}\Omega\times \ve^{-1}\Omega}\frac{u_0(x)^2|\phi(\ve x)-\phi(\ve y)|^2}{|x-y|^{n+2s}}\, dxdy\right)^{\frac{1}{2}}.\\
	\end{align*}
	Since $u_0$ is an extremal function for the constant $\mu_{\alpha}$, we obtain
	\begin{equation}\label{casi-casi}
	[u_{\ve}]_{s,\Omega}\le \mu_{\alpha}^{\frac{1}{2}}+\left(\int_{\ve^{-1}\Omega\times \ve^{-1}\Omega}\frac{u_0(x)^2|\phi(\ve x)-\phi(\ve y)|^2}{|x-y|^{n+2s}}\, dxdy \right)^{\frac{1}{2}}.
	\end{equation}
	We will show that 
	\begin{equation}
	\label{remained-term}
	\lim_{\ve \to 0} \int_{\ve^{-1}\Omega\times \ve^{-1}\Omega}\frac{u_0(x)^2|\phi(\ve x)-\phi(\ve y)|^2}{|x-y|^{n+2s}}\, dxdy=0.
	\end{equation}
	That will be a consequence of the Lebesgue Dominated convergence Theorem. Clearly, 
	$$
	\lim_{\ve \to 0} \chi_{\ve^{-1}\Omega\times \ve^{-1}\Omega}(x,y)\frac{u_0(x)^2|\phi(\ve x)-\phi(\ve y)|^2}{|x-y|^{n+2s}}= 0 \quad \text{ a.e. in } \R^n \times \R^n.
	$$
	To find the dominated function in $L^1(\R^n\times \R^n)$, we split the domain, and use \eqref{asymp}. Indeed, for every $0<\ve<1$,
	$$
	\frac{u_0(x)^2|\phi(\ve x)-\phi(\ve y)|^2}{|x-y|^{n+2s}} \le C\psi(x,y)\left(\chi_{\{|x|<1\}} u_0(x)^2 + \chi_{\{|x|\ge 1\}}\frac{1}{|x|^{2(n-2s)}}\right)=:\Psi(x,y),
	$$
	where $\psi(x,y)= \frac{1}{|x-y|^{n+2s-2}}\chi_{\{|x-y|<1 \}} + \frac{1}{|x-y|^{n+2s}}\chi_{\{|x-y|\ge1 \}}$. For the previous inequality, we have used 
	\begin{equation*}
	\frac{|\phi(\ve x)-\phi(\ve y)|^2}{|x-y|^{n+2s}}\le \begin{cases} \frac{C\ve^2}{|x-y|^{n+2s-2}} & \text{ if } |x-y|<1,\\
	\frac{C}{|x-y|^{n+2s}} & \text{ if } |x-y|\ge1. \\ 
	\end{cases}
	\end{equation*}

	Let us see that $\Psi \in L^1(\R^n\times \R^n)$.
	\begin{align*}
	\int_{\R^n\times\R^n} \Psi(x,y) \, dxdy &\le C\int_{|x|<1}u_0(x)^2 \int_{\R^n}\psi(x,y)\, dydx + C\int_{|x|\ge 1}\frac{1}{|x|^{2(n-2s)}}\int_{\R^n} \psi(x,y)\, dydx\\
	&\le C\int_{|x|<1}u_0(x)^2 dx + C\int_{|x|\ge 1}\frac{1}{|x|^{2(n-2s)}}dx\\
	&\le C\int_{|x|<1}\frac{u_0(x)^2}{|x|^{\frac{2\alpha}{2_{s,\alpha}}}} dx + C\\
	\end{align*}
	In the last step, we have used $n>4s$ in the second term. Then, apply H\"older inequality with $p=\frac{2_{s,\alpha}}{2}, p'=\frac{n-\alpha}{2s-\alpha}$ in the first term, to obtain
	\begin{align*}
	\int_{\R^n\times\R^n} \Psi(x,y) \, dxdy &\le C\left(\int_{\R^n} \frac{u_0^{2_{s,\alpha}}(x)}{|x|^{\alpha}}\, dx\right)^{\frac{2}{2_{s,\alpha}}} + C=C.
	\end{align*}
	Hence, \eqref{remained-term} holds. Consequently, from \eqref{casi-casi},
	$$
	\limsup_{\ve \to 0}[u_{\ve}]^2_{s,\Omega}\le \mu_{\alpha}.
	$$
	Then, \eqref{v_epsilon} becomes
	\begin{align*}
	\mu_{\alpha,\lambda}(\Omega) \le  \frac{1}{\|u_{\ve}\|_{s,\alpha,\Omega}^{2}}[u_{\ve}]_{s,\Omega}^2 + O(\ve^{2s}).
	\end{align*}
	Taking the limit $\ve\to 0$, we conclude $\mu_{\alpha,\lambda}(\Omega) \le \mu_{\alpha}$.
	
	(2) It follows from the definition \eqref{defi-mu-alpha-lambda}.
	
	(3) Consider $c:= \left(\int_{\Omega}\frac{1}{|x|^{\alpha}} \, dx\right)^{-\frac{1}{2_{s,\alpha}}}\in H^s(\Omega)$. Then, 
	$$
	\mu_{\alpha,\lambda}(\Omega) \le [c]_{s,\Omega}^2+\lambda \int_{\Omega}c^2\, dx= \lambda c^2 |\Omega|.
	$$
	Now, take the limit $\lambda \to 0$ to conclude (3).
\end{proof}

The next Corollary will be one of the main tools for proving Theorem \ref{main-theo}. It is a straightforward consequence of Lemma \ref{basic-lemma}.
\begin{corol}\label{coro} One of the following statements holds:
	\begin{itemize}
		\item[(1)] For every $\lambda>0$, we have the strict inequality $\mu_{\alpha,\lambda}(\Omega) <\mu_{\alpha}$, and $\lim_{\lambda \to \infty}\mu_{\alpha,\lambda}(\Omega) =\mu_{\alpha}$.
		\item[(2)] There exists $\bar{\lambda}>0$ such that $\mu_{\alpha,\lambda}(\Omega) =\mu_{\alpha}$ for every $\lambda\ge \bar{\lambda}$.
		\end{itemize}
\end{corol}

\section{Existence of extremal function.}\label{zeroinside}

We start this section with the second ingredient to prove Theorem \ref{main-theo}, which is a fractional Hardy-Sobolev type inequality. We follow ideas from \cite{Hashizume}, where the local version was studied. 

\begin{lema} \label{mu-positive}
Let $\Omega \subset \R^n$ be a bounded domain such that $0\in \Omega$. Then, for every $\ve>0$ there exists a positive constant $C_1(\ve)=C_1(\Omega,n,s, \ve)$ such that 
\begin{equation}\label{auxiliary-constant}
\frac{\mu_{\alpha}}{1+\ve}\left( \int_{\Omega} \frac{|u(x)|^{2_{s,\alpha}}}{|x|^{\alpha}}\, dx \right)^{\frac{2}{2_{s,\alpha}}} \le [u]_{s,\Omega}^2 +C_1(\ve) \int_{\Omega}|u|^2\, dx
\end{equation}
for every $u\in H^s(\Omega)$.
\end{lema}
\begin{proof}
Let $ \Omega_1 \subset \Omega_2 \subset \Omega$ be bounded sets  to be determined, such that $0\in \Omega_1$. Let $\phi \in C_c^{\infty}(\Omega)$ be such that $0\le \phi \le 1$ in $\Omega$, $\phi=1$ in $\Omega_1$, $\phi=0$ in $\Omega \setminus \Omega_2$. 
Consider
$$
\eta_1=\frac{\phi^2}{\phi^2 +(1-\phi)^2}, \quad \eta_2=\frac{(1-\phi)^2}{\phi^2+(1-\phi)^2}.
$$ 
Then, $\eta_1^{\frac{1}{2}} \in C_c^1(\Omega), \eta_2^{\frac{1}{2}} \in C^1(\Omega)$,  $\eta_1+\eta_2=1$, $\supp \eta_1 \subset \Omega_2 \subset \Omega, \, \, \supp \eta_2 \subset \R^n \setminus \Omega_1$. 
Let $u\in  H^s(\Omega)$. We consider $\eta_2^{\frac{1}{2}} u \colon \Omega \to \R$, by \cite[Lemma 5.3]{DiNezza-Palatucci-Valdinoci}, $\eta_2^{\frac{1}{2}} u\in H^s(\Omega)$, since $u\in H^s(\Omega)$ and $\eta_2^{\frac{1}{2}} \in C^{0,1}(\Omega)$. Moreover, $\|\eta_{2}^{\frac{1}{2}}u\|_{H^s(\Omega)} \le C(n,s,\Omega)\|u\|_{H^s(\Omega)}$.
By using the auxiliary functions $\eta_1,\eta_2$, we can split the main integral into two pieces and analyze them separately, as follows,
\begin{align*}
\mu_{\alpha}\left( \int_{\Omega} \frac{|u(x)|^{2_{s,\alpha}}}{|x|^{\alpha}}\, dx \right)^{\frac{2}{2_{s,\alpha}}}&\le  \mu_{\alpha}  \left(\sum_{i=1}^2  \left( \int_{\Omega} \frac{|\eta_i^{\frac{1}{2}} u|^{2_{s,\alpha}}}{|x|^{\alpha}}  \, dx \right)^{\frac{2}{2_{s,\alpha}}} \right) =: I_1+I_2.
\end{align*}
To estimate $I_1$, notice that we can use the fractional Hardy-Sobolev inequality given by $\mu_{\alpha}$ for $\eta_1^{\frac{1}{2}}u$, see \eqref{global-mu}. Thus, 
\begin{equation}\label{I-1}
I_1= \mu_{\alpha}  \left( \int_{\Omega} \frac{|\eta_1^{\frac{1}{2}} u|^{2_{s,\alpha}}}{|x|^{\alpha}}  \, dx \right)^{\frac{2}{2_{s,\alpha}}}=\mu_{\alpha}  \left( \int_{\R^n} \frac{|\eta_1^{\frac{1}{2}} u|^{2_{s,\alpha}}}{|x|^{\alpha}}  \, dx \right)^{\frac{2}{2_{s,\alpha}}} \le [\eta_1^{\frac{1}{2}}u]_s^2\\
\end{equation}
Notice that $\supp \eta_1 \subset \Omega$. Similarly to \eqref{casi-casi}, we obtain

\begin{align*}
[\eta_1^{\frac{1}{2}}u]_s
&\le \left(\int_{\Omega\times \Omega} \frac{|\eta_1^{\frac{1}{2}}(x)u(x)-\eta_1^{\frac{1}{2}}(y)u(y)|^2}{|x-y|^{n+2s}}\, dxdy + 2\int_{(\R^n \setminus \Omega)\times  \Omega} \frac{\eta_1(x)|u(x)|^2}{|x-y|^{n+2s}}\, dxdy\right)^{\frac{1}{2}}\\
\end{align*}
For the first term, we use \eqref{add-subtract} for $\eta_1^{\frac{1}{2}}u$ and Minkowski's inequality. 
For the second term, we proceed similar to Lemma \ref{cut-off-lemma} \eqref{cut-off-bound}, to get
\begin{align*}
[\eta_1^{\frac{1}{2}}u]_s&\le  \left(\int_{\Omega\times \Omega} \frac{\eta_1(y)|u(x)-u(y)|^2}{|x-y|^{n+2s}}\, dxdy\right)^{\frac{1}{2}} +C(\phi, n, s)\left( \int_{\Omega}|u|^2\, dx \right)^{\frac{1}{2}}
\end{align*}
which implies, by using $(a+b)^2 \le (1+\ve)a^2 +(1+\ve^{-1})b^2$ for every $\ve>0$, 
\begin{equation}\label{equ-eta_1-u}
[\eta_1^{\frac{1}{2}}u]_{s}^2\le (1+\ve)\int_{\Omega\times\Omega} \frac{\eta_1(y)|u(x)-u(y)|^2}{|x-y|^{n+2s}}\, dxdy+ C(\phi, n, s, \ve)\int_{\Omega}|u|^2\, dx.
\end{equation}

Therefore, taking into account \eqref{I-1}-\eqref{equ-eta_1-u}, we obtain
\begin{equation}\label{I_1}
I_1\le  (1+\ve)\int_{\Omega\times\Omega} \frac{\eta_1(y)|u(x)-u(y)|^2}{|x-y|^{n+2s}}\, dxdy +C(\phi, n, s, \ve)\int_{\Omega}|u(x)|^2 \, dx
\end{equation}
To analyze $I_2$, notice that $\eta_2=0$ in $\Omega_1$, so that
\begin{align*}
I_2&= \mu_{\alpha}  \left( \int_{\Omega} \frac{|\eta_2^{\frac{1}{2}} u|^{2_{s,\alpha}}}{|x|^{\alpha}}  \, dx \right)^{\frac{2}{2_{s,\alpha}}} = \mu_{\alpha}  \left( \int_{\Omega \setminus \Omega_1} \frac{|\eta_2^{\frac{1}{2}} u|^{2_{s,\alpha}}}{|x|^{\alpha}}  \, dx \right)^{\frac{2}{2_{s,\alpha}}}.
\end{align*}
Observe that $0\notin \supp \eta_2$. Denote by $d_1:=\dist(0,\partial \Omega_1)$. Thus, by H\"older's inequality with $p=\frac{n}{n-\alpha}, p'= \frac{n}{\alpha}$,
\begin{align*}
I_2&\le \mu_{\alpha}{d_1^{-\frac{2\alpha}{2_{s,\alpha}}}}  \left( \int_{\Omega \setminus \Omega_1}|\eta_2^{\frac{1}{2}} u|^{2_{s,\alpha}}  \, dx \right)^{\frac{2}{2_{s,\alpha}}} \le  \mu_{\alpha}{d_1^{-\frac{2\alpha}{2_{s,\alpha}}}}  \left( |\Omega\setminus \Omega_1|^{\frac{\alpha}{n}}\left(\int_{\Omega \setminus \Omega_1}|\eta_2^{\frac{1}{2}} u|^{2_{s}^*}  \, dx \right)^{\frac{n-\alpha}{n}} \right)^{\frac{2}{2_{s,\alpha}}}\\
&\le \mu_{\alpha}{d_1^{-\frac{2\alpha}{2_{s,\alpha}^*}}} |\Omega\setminus \Omega_1|^{\frac{2\alpha}{n 2_{s,\alpha}}} \left( \int_{\Omega \setminus \Omega_1}|\eta_2^{\frac{1}{2}} u|^{2_{s}^*}  \, dx  \right)^{\frac{2}{2_{s}^*}} \le \mu_{\alpha}{d_1^{-\frac{2\alpha}{2_{s,\alpha}}}} |\Omega\setminus \Omega_1|^{\frac{2\alpha}{n 2_{s,\alpha}}} \kappa^{-1}_{\Omega_1}[\eta_2^{\frac{1}{2}} u]_{s,\Omega}^2,
\end{align*}
where $\kappa_{\Omega_1}$ is given by
$$
\kappa_{\Omega_1}:= \inf\left\{ [ v]^2_{s,\Omega} \,  \colon v\in H^s(\Omega),\, v=0 \text{ in } \Omega_1, \, \int_{\Omega} |v|^{2_{s}^*} \, dx =1 \right\}.
$$
It will be enough to prove that 
\begin{equation}\label{choose-delta}
\mu_{\alpha}{d_1^{-\frac{2\alpha}{2_{s,\alpha}}}} |\Omega\setminus \Omega_1|^{\frac{2\alpha}{n 2_{s,\alpha}}} \kappa^{-1}_{\Omega_1} \le 1.
\end{equation}
Indeed, given $\delta>0$, choose $\Omega_1\subset \Omega$ such that $0\in \Omega_1$ and $|\Omega \setminus \Omega_1|<\delta$. Let $\Omega_0\subset \Omega$ be an open bounded set such that $0\in \Omega_0 \subset \Omega_1$. Then, $d_1\ge d_0:=\dist(0,\partial \Omega_0)$. Moreover, $\kappa_{\Omega_0}\le \kappa_{\Omega_1}$. 
Therefore,
\begin{align*}
\mu_{\alpha}{d_1^{-\frac{2\alpha}{2_{s,\alpha}}}} |\Omega\setminus \Omega_1|^{\frac{2\alpha}{n 2_{s,\alpha}}} \kappa^{-1}_{\Omega_1}&\le \mu_{\alpha}{d_0^{-\frac{2\alpha}{2_{s,\alpha}}}} |\Omega\setminus \Omega_1|^{\frac{2\alpha}{n 2_{s,\alpha}}} \kappa^{-1}_{\Omega_0} \\
&\le C\left(\Omega_0\right)|\Omega \setminus \Omega_1|^{\frac{2\alpha}{n 2_{s,\alpha}}}\le C\left(\Omega_0\right)\delta^{\frac{2\alpha}{n 2_{s,\alpha}}}.
\end{align*}
Let $\delta>0$ be such that $C(\Omega_0)\delta^{\frac{2\alpha}{n 2_{s,\alpha}}}<1$. Consequently, proceeding similar to the estimate of $[\eta_1^{\frac{1}{2}}u]_s$,  we obtain 
\begin{equation}
\label{I_2}I_2\le [\eta_2^{\frac{1}{2}}u]^2_{s,\Omega}\le (1+\ve)\int_{\Omega\times\Omega}\frac{\eta_2(y)|u(x)-u(y)|^2}{|x-y|^{n+2s}}\, dxdy +C(\phi, n, s, \ve)\int_{\Omega}|u(x)|^{2}\, dx.
\end{equation}
By \eqref{I_1},\eqref{I_2} and the fact that $\eta_1+\eta_2=1$, we conclude \eqref{auxiliary-constant}, where the constant only depends on $\Omega_0,\phi, n, s$ and $\ve$, then $C(\Omega,n,s, \ve)=C(\ve)$.
\end{proof}

Combining Lemmas \ref{basic-lemma} and \ref{mu-positive}, we get the next proposition which gives (non)existence of an extremal function for $\mu_{\alpha, \lambda}(\Omega)$, depending on the relation with the global constant in $\R^n$, i.e. $\mu_{\alpha}$.

\begin{prop}\label{key-prop-theo-1} Let $\lambda>0$ and $\Omega \subset \R^n$ be a bounded domain such that $0\in \Omega$.
\begin{itemize}
\item[(1)] If $\mu_{\alpha,\lambda}(\Omega)< \mu_{\alpha}$, then $\mu_{\alpha,\lambda}(\Omega)$ is attained. 
\item[(2)] If there exists a $\bar{\lambda}>0$ such that $\mu_{\alpha,\bar{\lambda}}(\Omega)=\mu_{\alpha}$, then for every $\lambda>\bar{\lambda}$, $\mu_{\alpha,\lambda}(\Omega)$ is not attained.
\end{itemize}
\end{prop}
\begin{proof}
(i) Let $\{ u_k\}_{k\in \N}\subset H^s(\Omega)$ be a minimizing sequence for $\mu_{\alpha,\lambda}(\Omega)$, that is, 
$$
\int_{\Omega}\frac{|u_k|^{2_{s,\alpha}}}{|x|^{\alpha}}\, dx =1 \, \text{ for every } k\in \N, \text{ and } \lim_{k\to\infty}\left([u_k]_{s,\Omega}^2+\lambda\int_{\Omega}|u_k|^2\, dx \right)= \mu_{\alpha, \lambda}(\Omega). 
$$
Then, $\{ u_k\}_{k\in \N}$ is bounded in $H^s(\Omega)$. Therefore, up to a subsequence, we can assume that 
\begin{itemize}
	\item[] $u_k\cd u$ weakly in $H^s(\Omega)$, 
	\item[] $u_k\to u$ strongly in $L^p(\Omega)$ for $1\le p<2_s^*=\frac{2n}{n-2s}$, see \cite[Theorem 4.54]{Demengel-Demengel}, 
	\item[] $u_k\to u$ a.e. in $\Omega$
\end{itemize}


Let us see that $u\not \equiv 0$. We proceed by contradiction. Assume $u\equiv 0$ a.e. in $\Omega$ and let $\ve>0$. By \eqref{auxiliary-constant}, we get
\begin{align*}
\frac{\mu_{\alpha}}{1+\ve}&=\frac{\mu_{\alpha}}{1+\ve} \left( \int_{\Omega}\frac{|u_k|^{2_{s,\alpha}}}{|x|^{\alpha}}\, dx\right)^{\frac{2}{2_{s,\alpha}}} \le [u_k]_{s,\Omega}^2 +C(\ve)\int_{\Omega}|u_k|^2\, dx\\
\end{align*}
which implies
\begin{equation}\label{ojo}
\frac{\mu_{\alpha}}{1+\ve} \le \mu_{\alpha,\lambda}(\Omega)+o(1)+ (C(\ve)-\lambda)\int_{\Omega}|u_k|^2\, dx.
\end{equation}
By taking the limit in $k$, we get $\frac{\mu_{\alpha}}{1+\ve}\le \mu_{\alpha,\lambda}(\Omega)$ for every $\ve>0$. Thus, letting $\ve\to 0$, we obtain $\mu_{\alpha} \le \mu_{\alpha,\lambda}(\Omega)$ which is a contradiction. Therefore, $u \not \equiv 0$ in $\Omega$.
By Brezis-Lieb Theorem \cite{Brezis-Lieb}, we know that
$$
 \int_{\Omega}\frac{|u_k|^{2_{s,\alpha}}}{|x|^{\alpha}}\, dx= \int_{\Omega}\frac{|u|^{2_{s,\alpha}}}{|x|^{\alpha}}\, dx+ \int_{\Omega}\frac{|u_k-u|^{2_{s,\alpha}}}{|x|^{\alpha}}\, dx +o(1),
$$
from it follows that
\begin{align*}
1&= \left( \int_{\Omega}\frac{|u_k|^{2_{s,\alpha}}}{|x|^{\alpha}}\, dx\right)^{\frac{2}{2_{s,\alpha}}} = \left( \int_{\Omega}\frac{|u|^{2_{s,\alpha}}}{|x|^{\alpha}}\, dx+ \int_{\Omega}\frac{|u_k-u|^{2_{s,\alpha}}}{|x|^{\alpha}}\, dx +o(1)   \right)^{\frac{2}{2_{s,\alpha}}}\\
&\le \left( \int_{\Omega}\frac{|u|^{2_{s,\alpha}}}{|x|^{\alpha}}\, dx \right)^{\frac{2}{2_{s,\alpha}}}+\left( \int_{\Omega}\frac{|u_k-u|^{2_{s,\alpha}}}{|x|^{\alpha}}\, dx \right)^{\frac{2}{2_{s,\alpha}}}+o(1)\\
&\le \frac{1}{\mu_{\alpha,\lambda}(\Omega)}\left([u]^2_{s,\Omega}+\lambda\int_{\Omega}|u|^2\, dx\right)\\
&+ \frac{1}{\mu_{\alpha,\lambda}(\Omega)}\left([u_k-u]^2_{s,\Omega}+\lambda\int_{\Omega}|u_k-u|^2\, dx\right) +o(1)\\
&=\frac{1}{\mu_{\alpha,\lambda}(\Omega)}\left([u_k]^2_{s,\Omega}+\lambda\int_{\Omega}|u_k|^2\, dx\right) +o(1) \\
&=1+o(1).
\end{align*}
Notice that we have used that 
\begin{align*}
|(u_k-u)(x)-(u_k-u)(y)|^2&= |u_k(x)-u_k(y)|^2+|u(x)-u(y)|^2\\
&-2(u_k(x)-u_k(y))(u(x)-u(y)),
\end{align*}
implies that
\begin{align*}
[u]^2_{s,\Omega}+[u_k-u]^2_{s,\Omega}&\le [u_k]^2_{s,\Omega}+ 2[u]^2_{s,\Omega}- 2\int_{\Omega \times\Omega}\frac{(u_k(x)-u_k(y))(u(x)-u(y))}{|x-y|^{n+2s}}\, dxdy \\
&= [u_k]^2_{s,\Omega}+o(1),
\end{align*}	
due to the weakly convergence $u_k\cd u$ in $H^s(\Omega)$.
As a consequence, there exists the following limit
\begin{align*}
1&=\lim_{k\to \infty} \left( \int_{\Omega}\frac{|u|^{2_{s,\alpha}}}{|x|^{\alpha}}\, dx+ \int_{\Omega}\frac{|u_k-u|^{2_{s,\alpha}}}{|x|^{\alpha}}\, dx  \right)^{\frac{2}{2_{s,\alpha}}}\\
&=\lim_{k\to\infty}\left[ \left( \int_{\Omega}\frac{|u|^{2_{s,\alpha}}}{|x|^{\alpha}}\, dx \right)^{\frac{2}{2_{s,\alpha}}}+ \left(\int_{\Omega}\frac{|u_k-u|^{2_{s,\alpha}}}{|x|^{\alpha}}\, dx \right)^{\frac{2}{2_{s,\alpha}}} \right].
\end{align*}
Since $u\not \equiv 0$, we conclude that $u_k\to u$ strongly in $L^{2_{s,\alpha}}(\Omega, |x|^{-\alpha}dx)$, and 
$$
\int_{\Omega}\frac{|u|^{2_{s,\alpha}}}{|x|^{\alpha}}\, dx=1,$$
which implies that $\mu_{\alpha,\lambda}(\Omega)$ is attained by $u$. 

\medskip

(ii) Let $\lambda>\bar{\lambda}$. Assume that there exists a function $u\in H^s(\Omega)$ which is a minimizer to $\mu_{\alpha,\lambda}(\Omega)$. Then, 
$$
\mu_{\alpha,\lambda}(\Omega)=[u]^2_{s,\Omega}+\lambda\int_{\Omega}|u|^2\, dx >[u]^2_{s,\Omega}+\bar{\lambda}\int_{\Omega}|u|^2\, dx \ge \mu_{\alpha,\bar{\lambda}}(\Omega)=\mu_{\alpha}\ge \mu_{\alpha,\lambda}(\Omega),
$$
where we have used (1) from Lemma \ref{basic-lemma} in the last inequality. This contradiction finishes the proof.  
\end{proof}

Now, we are in condition to prove Theorem \ref{main-theo}.
\begin{proof}[Proof of Theorem \ref{main-theo}] We define 
	$\lambda_* = \inf\{ \lambda>0 \colon \mu_{\alpha, \lambda}(\Omega) = \mu_{\alpha} \} \in (0, \infty]$. The proof follows from Corollary \ref{coro} and Proposition \ref{key-prop-theo-1}. 
\end{proof}

\section*{Acknowledgments}
The author wants to thank Prof. Marco Squassina 
for drawing her attention to this topic and for helpful discussions. This project has received funding from the European Union's Horizon 2020 research and innovation program under the Marie Skłodowska-Curie grant agreement No 777822.

\bibliographystyle{amsplain}
\bibliography{biblio}

\end{document}